\documentclass[A4paper]{article}%
\usepackage{graphicx}
\usepackage{amsmath,amssymb,amsfonts}
\usepackage{theorem}
\usepackage{color}
\usepackage{hyperref}
\usepackage[font={small, it}]{caption}
\usepackage[all]{xy}%
\usepackage{amsmath}%
\usepackage{amsfonts}%
\usepackage{amssymb}
\providecommand{\U}[1]{\protect \rule{.1in}{.1in}}
\theoremstyle{change}
\newtheorem{definition}{Definition:}[section]
\newtheorem{proposition}[definition]{Proposition:}
\newtheorem{theorem}[definition]{Theorem:}
\newtheorem{lemma}[definition]{Lemma:}
\newtheorem{corollary}[definition]{Corollary:}
{\theorembodyfont{\rmfamily}
	\newtheorem{remark}[definition]{Remark:}
}
{\theorembodyfont{\rmfamily}
	
}
\newenvironment{proof}
{{\bf Proof:}}
{\qquad \hspace*{\fill} $\Box$}

\newcommand{\fg}{\mathfrak{g}}

\newcommand{\fn}{\mathfrak{n}}

\newcommand{\fz}{\mathfrak{z}}
\newcommand{\fh}{\mathfrak{h}}

\newcommand{\Ad}{\operatorname{Ad}}
\newcommand{\ad}{\operatorname{ad}}

\newcommand{\inner}{\operatorname{int}}

\newcommand{\fix}{\operatorname{fix}}

\newcommand{\rme}{\mathrm{e}}

\newcommand{\EC}{\mathcal{E}}
\newcommand{\LC}{\mathcal{L}}

\newcommand{\OC}{\mathcal{O}}

\newcommand{\UC}{\mathcal{U}}

\newcommand{\NC}{\mathcal{N}}
\newcommand{\HC}{\mathcal{H}}

\newcommand{\XC}{\mathcal{X}}
\newcommand{\DC}{\mathcal{D}}

\newcommand{\N}{\mathbb{N}}

\newcommand{\R}{\mathbb{R}}

\begin{document}

\title{Existence and uniqueness of control sets with a nonempty interior for linear control systems on solvable groups}
\author{Adriano Da Silva\thanks{Supported by Proyecto UTA Mayor Nº 4781-24} \\
	Departamento de Matem\'atica,\\Universidad de Tarapac\'a - Iquique, Chile.}
\date{\today}
\maketitle

\begin{abstract}
In this paper, we obtain weak conditions for the existence of a control set with a nonempty interior for a linear control system on a solvable Lie group. We show that the Lie algebra rank condition together with the compactness of the nilpotent part of the generalized kernel of the drift are enough to assure the existence of such a control set. Moreover, this control set is unique and contains the whole generalized kernel in its closure.
\end{abstract}

	

\section{Introduction}

In the dynamical study of control systems, the concept of control sets plays an essential role. Several dynamical entities of the systems, such as equilibrium points, recurrent points, periodic and bounded orbits, etc., appear naturally inside them (see \cite[Chapter 4]{FCWK}). Also, the interior of a control set is where exact controllability takes place. On the other hand, by the Equivalence Theorem of P. Jouan \cite{JPh1}, linear control systems on Lie groups and homogeneous spaces appear as a classifying family for more general control-affine systems on connected manifolds, implying that their understanding is worth it in more general contexts.

Due to the previous, in this paper, we extend the results in \cite{DS1} by showing that the Lie algebra rank condition and the compactness of the nilpotent part of the generalized kernel of the drift are sufficient conditions to assure the existence of a control set with a nonempty interior of a linear control system on a solvable Lie group. 
The compactness assumption allows us to simplify our work to the study of the Cartesian product of a linear control system with nilpotent drift on a vectorial space by a control-affine system on a connected, simply connected, nilpotent Lie group. Moreover, once the system on the nilpotent component admits one such control, it guarantees the existence of a control set with a nonempty interior, indicating that the primary focus must be on this system. In order to understand the system on the nilpotent part, a detailed analysis of a class of diffeomorphisms, induced by group automorphisms, has to be done. Such an analysis allows us to construct points living in the interior of their own positive orbit, which implies the existence of control sets containing them.

The paper is structured as follows: In Section 2 we introduce linear vector fields and study a decomposition induced by their generalized kernel and the nilradical of the group/algebra. In the sequence, we study a type of diffeomorphism induced by an automorphism on a nilpotent Lie group. These diffeomorphisms are intrinsically connected with the solutions of linear systems, as we see in the subsequent sections. To conclude, we introduce control-affine systems on arbitrary manifolds and their control sets. We present here some general properties about positive orbits, control sets, and conjugations. In Section 3 we start the study of the control sets with a nonempty interior of linear control systems. We start by showing that for such a study, one can disregard any compact subgroup that is invariant by the drift, simplifying our result to connected and simply connected spaces. In the sequence we show that, if the nilpotent part of the generalized kernel (at the group level) is a compact subgroup, then the linear control system admits at most one control set with a nonempty interior. In Section 4, we tackle the existence problem. We start with the study of a particular class of control-affine systems on a connected, simply connected nilpotent Lie group. Under the Lie algebra rank condition, we prove that this class admits a control set $D$ with a nonempty interior that contains the origin of the group in its closure. In sequence, we consider the Cartesian of this system with a linear control system on a vectorial space $V$, whose associated drift is nilpotent. For this Cartesian, we prove that $V\times D$ is a control set.To conclude, we prove that, up to the quotient by compact subgroups, a linear control system on a solvable group is equivalent to a system on the Cartesian product considered previously. In particular, the linear control system admits a control set with a nonempty interior if it satisfies the Lie algebra rank condition and the nilpotent part of the generalized kernel is compact.

%{\bf Notations:}{\color{red}
%Let n be a nilpotent Lie algebra and identify it with the connected, simply connected Lie group (n, ∗), where
%the product is given by the BCH formula. Such an identification allows us to work with the vectorial structure
%of n and the group product on the same space.}

\section{Preliminares}

In this section we introduce some basic concepts needed throughout the paper. We also use this section to prove some technical lemmas, which will be very useful ahead.

\subsection{Linear vector fields and their derivations}

Let $G$ be a connected Lie group with Lie algebra $\mathfrak{g}$ identified with the set of right-invariant vector fields on $G$. A vector field $\mathcal{X}$ on $G$ is said to be {\bf linear} if for any $Y\in\fg$ it holds that 
$$[\XC, Y]\in \fg\;\;\;\mbox{ and }\;\;\;\XC(e)=0.$$
Following \cite[Theorem 1]{JPh1}, a linear vector field $\XC$ is complete, and its associated flow $\{\varphi_t\}_{t\in \R}$ is a $1$-parameter subgroup of $\mathrm{Aut}(G)$ satisfying
\begin{equation}
\label{derivativeonorigin} (d\varphi_{t})_{e}=\mathrm{e}^{t\mathcal{D}}\hspace{.5cm}\mbox{ and }\hspace{.5cm}\varphi_t(\exp X)=\exp(\rme^{t\DC}X), \hspace{.5cm}\forall t\in\R, X\in\fg,
\end{equation}
where $\DC:\fg\rightarrow\fg$ is the derivation defined by $\DC(Y):=[\XC, Y]$, and $\rme^{t\DC}$ its matrix exponential.

Let $\DC$ be a derivation of $\fg$. We say that $\DC$ is {\bf nilpotent} if $\DC^n\equiv0$ for some $n\in\N$, and $\DC$ is {\bf elliptic} ({\bf resp. hyperbolic}) if it is semisimple and its eigenvalues are pure imaginary ({\bf resp. real}). The {\bf (additive) Jordan decomposition} of $\DC$ is given by the
$$\DC=\DC_{\HC}+\DC_{\EC}+\DC_{\NC}, \hspace{.5cm}\mbox{ where
}\hspace{.5cm}\DC, \DC_{\HC}, \DC_{\EC}\mbox{ and }\DC_{\NC}\mbox{ commute, }$$
and $\DC_{\HC}$ is hyperbolic, $\DC_{\EC}$ is elliptic, and $\DC_{\NC}$ is nilpotent. Moreover, $\DC_{\HC}, \DC_{\EC}$, and $\DC_{\NC}$ are also derivations.

By the relation between linear vector fields and derivations, we say that a linear vector field $\XC$ is {\bf hyperbolic, elliptic}, or {\bf nilpotent} if its associated derivation $\DC$ is one of the respective types. By the results in \cite{DSAyPH}, the Jordan decomposition of $\DC$ implies the commutative decomposition of $\XC=\XC_{\HC}+\XC_{\EC}+\XC_{\NC}$ (also called the Jordan decomposition) into linear vector fields $\XC_{\HC}, \XC_{\EC}$, and $\XC_{\NC}$ whose associated derivations $\DC_{\HC}, \DC_{\EC}$, and $\DC_{\NC}$ are, respectively, hyperbolic, elliptic, and nilpotent.

\begin{remark}
\label{elliptical}
By \cite[Theorem 2.4]{DSAyPH}, any flow $\{\varphi_S\}_{S\in\R}$ of an elliptical vector field $\XC$ on $G$ is a flow of isometries for some left-invariant metric $\varrho$ on $G$. In particular, for any $S>0$ and $g\in G$, the sequence $\{\varphi_{nS_0}(g)\}_{n\in\N}$ admits a subsequence converging to $g$.

In fact, since $\{\varphi_S\}_{S\in\R}$ is a flow of isometries, the sequence $\{\varphi_{nS_0}(g)\}_{n\in\N}$ is bounded, and hence, there exists a subsequence $ \{\varphi_{n_kS_0}(h)\}_{k\in\N}$ that converges and is, in particular, a Cauchy sequence. Therefore, for any $p\in\N$ there exists $k_p\in\N$ such that
$$k\geq k_p\hspace{.5cm}\implies\hspace{.5cm}\varrho\left(\varphi_{(n_k-n_{k_p})S_0}(h), h\right)=\varrho\left(\varphi_{n_kS_0}(h), \varphi_{n_{k_p}}
(h)\right)<\frac{1}{p}.$$
Therefore, we can construct a sequence $ \{\varphi_{n_{k_p}S_0}(h)\}_{p\in\N}$ such that 
$$n_{k_p}\rightarrow +\infty \hspace{.5cm}\mbox{ and }\hspace{.5cm}\varphi_{n_{k_p}}(h)\rightarrow h, \hspace{.5cm}\mbox{ as }\hspace{.5cm}p\rightarrow+\infty.$$
\end{remark}

Let us now consider the {\bf generalized kernel}
$$\fg_0:=\{X\in\fg; \DC^nX=0, \mbox{ for some }n\in\N\},$$
of the derivation $\DC$. By considering the derivation $\DC^{\HC, \EC}=\DC^{\HC}+\DC^{\EC}$, it is not hard to see that $\fg_0=\ker\DC^{\HC, \EC}$, and hence, $\fg^0$ is a subalgebra of $\fg$. Furthermore, if $G_0$ is the connected subgroup associated to $\fg_0$ (also called generalized kernel), it holds that
$$\fg_0=\ker\DC^{\HC, \EC}\hspace{.5cm}\implies\hspace{.5cm}G_0=\left(\fix(\varphi^{\HC, \EC})\right)_1,$$ showing that $G_0$ is a closed subgroup. Let $\fn$ be the nilradical of $\fg$. Since $\DC$ restricts to $\fn$, we can consider the generalized kernels $\fn_0$ and $N_0$ of the restriction. On the other hand, if $N$ is the nilradical in $G$, then $N\cap G_0$ is the {\bf nilpotent part of the generalized kernel}. The next result highlights the importance of these subgroups.

\begin{lemma}
\label{conj}
Let $\XC$ be a linear vector field on a solvable Lie group $G$ and $\fn$ be the nilradical of $\fg$. It holds:

\begin{itemize}
 \item[1.] The Lie algebra $\fg$ satisfies $\fg=\fn+\fg_0$;
 \item[2.] $N_0=N\cap G_0$ is a closed subgroup of $G_0$, and $G_0/N_0$ is abelian.
 \item[3.] The map
$$\psi:G_0\times_{\rho}\fn\rightarrow G,\hspace{1cm} (g, X)\mapsto \rme^X g,$$
is a surjective homomorphism that conjugates $\XC$ the linear vector fields $\XC|_{G_0}\times\DC|_{\fn}$, where the semi-direct product $G_0\times_{\rho}\fn$ is given by the map $\rho(g):=\Ad(g)|_{\fn}$. Moreover, $\psi$ is an isomorphism if and only if $N_0=\{e\}$.
\end{itemize}
%{\color{red}NECESSARIO?
%\bigskip

%\item[4.] If $\fn_0$ is an ideal of $\fn$, the map $\psi$ factors to an isomorphism  $\widehat{\psi}:G_0/N_0\times_{\hat{\rho}}\fn/\fn_0\rightarrow G/N_0$ that conjugate the linear vector field $\widehat{\XC}$ to the linear vector $\widehat{\XC}|_{G_0/N_0}\times\widehat{\DC}$, where $\widehat{\XC}$ is induced by $\XC$ on $G/N_0$, $\widehat{\XC}|_{G_0/N_0}=\widehat{\XC|_{G_0}}$ the restriction of $\widehat{\XC}$ to $G_0/N_0$ and $\widehat{\DC}$ induced on $\fn/\fn_{0}$ by $\DC|_{\fn}$.

%}

\end{lemma}

\begin{proof} 1. Since the complexification of $\DC^{\HC, \EC}$ is diagonalizable\footnote{In fact, $\DC^{\HC, \EC}$ is the semisimple part of $\DC$ in the Jordan–Chevalley decomposition.}, $\fg$ can be decomposed into invariant subspaces $W_{\lambda, \mu}$ such that
\begin{itemize}
 \item[(i)] $W_{\lambda, 0}=\langle X\rangle$ and $\DC^{\HC, \EC}X=\lambda X$ or
 \item[(ii)] $W_{\lambda, \mu}=\langle X, Y\rangle$ and $\DC^{\HC, \EC}X=\lambda X+\mu Y$ and $\DC^{\HC, \EC}Y=-\mu X+\lambda Y$.
\end{itemize}
Therefore,
$$\lambda^2+\mu^2\neq 0\hspace{.5cm}\implies\hspace{.5cm}W_{\lambda, \mu}\subset\mathrm{Im}(\DC^{\HC, \EC})\subset\fn,$$
where the last equality comes from the fact that the image of any derivation of a solvable algebra is contained in the nilradical \cite[Proposition 2.18]{SM1}. Therefore, 
$$\fg=\bigoplus_{\lambda^2+\mu^2\neq 0}W_{\lambda, \mu} \oplus \ker(\DC^{\HC, \EC})\subset \fn+\ker(\DC^{\HC, \EC})=\fn+\fg_0\subset\fg,$$
as stated.\bigskip

2. Since
$$G_0=\fix(\varphi^{\HC, \EC})_1\hspace{.5cm}\mbox{ and }\hspace{.5cm}N_0=\fix(\varphi|_N^{\HC, \EC})_1,$$
it holds that $N_0\subset G_0\cap N$. On the other hand, the inclusion 
$$G_0\cap N\subset\fix(\varphi|_N^{\HC, \EC}),$$
certainly holds. Hence, the fact that the exponential map of a nilpotent group is a covering map, together with relation (\ref{derivativeonorigin}), implies that $\fix(\varphi|_N^{\HC, \EC})$ is a connected subgroup, implying that the desired equality holds.

Since on solvable algebras $[\fg, \fg]\subset\fn$, it holds that
$$[\fg_0, \fg_0]\subset\fg_0\cap\fn=\fn_0,$$
showing that $\fn_0$ is an ideal of $\fg_0$ containing $[\fg_0, \fg_0]$. Consequently, $N_0$ is a normal subgroup of $G_0$, and $G_0/N_0$ is abelian.
\bigskip

3. Since $N$ is a normal subgroup of $G$, Lemma 3.1 of \cite{Wuns} implies that
$$\exp(\fg)\subset N\exp(\fg_0).$$
By the connectedness of $G$, the previous implies that any element of $x\in G$ can be written as $x=hg$ with $h\in N$ and $g\in G_0$. Since the exponential map of a nilpotent group is surjective, there exists $X\in\fn$ such that $h=\rme^X$, and hence,
$$x=hg=\rme^{X}g=\psi(g, X),$$
showing that $\psi$ is a surjective map. On the other hand, for any $(g_1, X_1), (g_2, X_2)\in G_0\times_{\rho}\fn$, it holds that
$$\psi((g_1, X_1)(g_2, X_2))=\psi(g_1g_2, X_1*\rho(g_1)X_2)=\rme^{X_1*\rho(g_1)X_2}g_1g_2$$
$$\rme^{X_1}\rme^{\Ad(g_1)X_2}g_1g_2=\rme^{X_1}g_1\rme^{X_2}g_2=\psi(g_1, X_1)\psi(g_2, X_2),$$
 showing that $\psi$ is a homomorphism. For the conjugation property, let us note that
    $$\psi(\varphi_
 S|_{G_0}(g), \rme^{S\DC|_{\fn}}X)=\exp\left(\rme^{S\DC|_{\fn}}X\right)\varphi_S|_{G_0}(g)=\exp\left(\rme^{S\DC}X\right)\varphi_S(g)=\varphi_S(\rme^X)\varphi_S(g)=\varphi_S(\rme^Xg)=\varphi_S(\psi(g, X)),$$
 which by derivation implies the conjugation between the associated vector fields. To conclude, $\psi$ is an isomorphism if and only if $\ker\psi=\{(e, 0)\}$. By a straightforward calculation, we obtain that
 $$\ker\psi=\{(g, X)\in N_0\times\fn_0; \;g^{-1}=\rme^X\}.$$
Therefore, $N_0=\{e\}$ implies $\ker\psi=\{(e, 0)\}$. Reciprocally, since $(\rme^{X}, X)\in \ker\psi$ for any $X\in\fn_0$, we conclude that $\ker\psi=\{(e, 0)\}$ implies $N_0=\{e\}$, concluding the proof.
\end{proof}
    
%\bigskip

%{\color{red}NECESSARIO?
%\bigskip

%4. It is enough to show that the map
%$$\pi:G_0\times_{\rho}\fn\rightarrow G_0/N_0\times_{\widehat{\rho}}\fn/\fn_0, \hspace{.5cm}(g, X)\mapsto (gN_0, x*\fn_0),$$
%is a homomorphism satisfying $\ker\psi\subset\ker\pi$. Now, for any $(g_1, X_1), (g_2, X_2)\in G_0\times_{\rho}\fn$, it holds that 
%$$\pi((g_1, X_1)(g_2, X_2))=\pi(g_1g_2, X_1*\rho(g_1)X_2)=(g_1g_2N_0, X_1*\rho(g_1)X_2*\fn_0)$$
%$$=\left(g_1N_0g_2N_0, (X_1*\fn_0)*
%(\rho(g_1)X_2*\fn_0)\right)=\left(g_1N_0g_2N_0, (X_1*\fn_0)*\widehat{\rho}(g_1N_0)(X_2*\fn_0)\right)$$
%$$=\left(g_1N_0, X_1*\fn_0\right)\left(g_2N_0, X_2*\fn_0\right)=\pi(g_1, X_1)\pi(g_2, X_2),$$
%showing that $\pi$ is a homomorphism. On the other hand, 
%$$(g, X)\in\ker\psi\iff \rme^Xg=e\iff N\ni\rme^X=g^{-1}\in G_0\implies (g, X)\in N_0\times \fn_0=\ker\pi,$$
%concluding the proof.}

\begin{remark}

The previous results show us that any linear vector field $\XC$ on a connected solvable Lie group with trivial $N_0$ is equivalent to the Cartesian of a nilpotent linear vector field on an abelian group by an invertible derivation on a nilpotent Lie algebra. We will use this fact in order to simplify linear control systems ahead.
\end{remark}

%{\color{blue}

%If $\fu_0:\{X\in \fu; \DC^nX=0, \mbox{ for some }n\in\N\}$ is an ideal, then
%$$f:\fu/\fu_0\rightarrow G/G_0, %\hspace{.5cm} X*\fu_0\mapsto \rme^XG_0,$$
%is a diffeormorphism.

%Surjectivity is direct from the fact that $\psi$ is surjective. It is well defined, since $\fu_0$ is an ideal!! (note that, $\fu/\fu_0$ has a Lie group structure. However, the map $f$ is NOT an isomorphism!!)

%}

\subsection{Maps induced by automorphisms}

In this section we study a class of diffeomorphisms induced by automorphisms of the groups. Such a class is intrinsically connected with the solutions of affine control systems on nilpotent Lie groups and will be of help in the proof of the main results.

Let $\fn$ be a nilpotent Lie algebra, and consider $(\fn, *)$ to be its connected, simply connected, nilpotent Lie group. Let $\varphi\in\mathrm{Aut}(\fn)$ and define the map
$$f_{\varphi}:\fn\rightarrow \fn, \hspace{.5cm}x\in \fn\mapsto f_{\varphi}(x):=x*\varphi(x)^{-1}.$$
The next proposition tells us about the continuity of the inverse of the map $\varphi\mapsto f_{\varphi}$.

\begin{lemma}
\label{lemma}
    Let $(\fn, *)$ be the simply connected, connected nilpotent Lie group with Lie algebra $\fn$. If $\varphi\in\mathrm{Aut}(\fn)$, then 
$$\det (\varphi-1)\neq 0\hspace{.5cm}\implies\hspace{.5cm}f_{\varphi}(x):=x*\varphi(x^{-1})\hspace{.5cm}\mbox{ is a diffeomorphism}.$$
\end{lemma}

\begin{proposition}
\label{continverse}
    Let us consider $\varphi_n, \varphi\in\mathrm{Aut}(\fn)$ with  $\varphi_n\rightarrow\varphi$ and $x_n, x\in \fn$. If $\det(1-\varphi)\neq 0$, then 
    $$f_{\varphi_n}(x_n)\rightarrow f_{\varphi}(x)\hspace{.5cm}\implies\hspace{.5cm}x_n\rightarrow x.$$
\end{proposition}

\begin{proof} Let us first assume that $(\fn, *)$ is an abelian, simply connected group. In this case, the map $f_{\varphi}$ coincides with the linear map $1-\varphi$. 
Let $T:\fn\rightarrow \fn$ be a linear map and define
$$\|T\|_{\mathrm{min}}:=\min\{|Tx|, \;|x|=1\}.$$
For all $x\in \fn$ it holds that $|Tx|\geq \|T\|_{\mathrm{min}}|x|$. Moreover, if $\det T\neq 0$ then $\|T\|_{\mathrm{min}}>0$ and 
$$T_n\rightarrow T\hspace{.5cm}\implies\hspace{.5cm}\|T_n\|_{\mathrm{min}}\rightarrow \|T\|_{\mathrm{min}}.$$ 
Therefore, for $n\in\N$ large enough,
$$|f_{\varphi_n}(x_n)|=|(1-\varphi_n)x_n|\geq \|(1-\varphi_n)\|_{\mathrm{min}}|x_n|\geq \frac{1}{2}\|(1-\varphi)\|_{\mathrm{min}}|x_n|,$$
Since $(f_{\varphi_n}(x_n))_{n\in\N}$ is convergent, we conclude that $(x_n)_{n\in\N}$ is bounded. Now, if $x_{n_k}\rightarrow y$, then
$$f_{\varphi_{n_k}}(x_{n_k})=x_{n_k}-\varphi_{n_k} x_{n_k}\rightarrow y-\varphi y\hspace{.5cm}\implies\hspace{.5cm}(1-\varphi)y=(1-\varphi)x\hspace{.5cm}\implies\hspace{.5cm} x=y.$$
Therefore, $(x_n)_{n\in\N}$ is bounded and has a unique accumulation point, which implies that $x_n\rightarrow x$ as desired.

Let us now proceed to prove the result by induction in the dimension of $\fn$. If $\dim \fn=1$, $\fn$ is abelian, and the result is true by the previous case. Let $(\fn, *)$ be the connected, simply connected nilpotent Lie group with Lie algebra $\fn$ and $\dim \fn=n$ and assume that the result holds true for any connected, simply connected nilpotent Lie group with smaller dimension.

By the hypothesis on $\fn$, the connected, simply connected nilpotent groups associated with the Lie algebras $\fz(\fn)$ and $\widehat{\fn}:=\fn/\fz(\fn)$ have dimension smaller than $n$. Since $\fz(\fn)$ is invariant by automorphisms, there exist
$\widehat{\varphi}_n, \widehat{\varphi}\in\mathrm{Aut}(\widehat{\fn})$ satisfying
$$\pi\circ \varphi=\widehat{\varphi}\circ\pi\hspace{.5cm}\mbox{ and }\hspace{.5cm}\pi\circ \varphi_n=\widehat{\varphi}_n\circ\pi,$$
where $\pi:\fn\rightarrow \widehat{\fn}$ is the canonical projection. Note that if $X\in \fn$ is such that
$$\pi((1-\varphi)X)=(1-\widehat{\varphi})\pi(X)=0\hspace{.5cm}\implies\hspace{.5cm}(1-\varphi)X=Y, \hspace{.5cm}\mbox{ for some }\hspace{.5cm}Y\in \fz(\fn).$$
However,
$$\det(1-\varphi)\neq 0\hspace{.5cm}\implies\hspace{.5cm} \det(1-\varphi|_{\fz(\fn)})\neq 0,$$
and hence, $Y=(1-\varphi)Z$ for some $Z\in \fz(\fn)$. Therefore,
$$(1-\varphi)X=(1-\varphi)Z\hspace{.5cm}\implies\hspace{.5cm}X=Z\in \fz(\fn)\hspace{.5cm}\implies\hspace{.5cm}\pi(X)=0,$$
showing that $\det(1-\widehat{\varphi})\neq 0$. Since $\pi\circ f_{\varphi}=f_{\widehat{\varphi}}\circ\pi$,
we conclude by the inductive hypothesis that 
$$f_{\widehat{\varphi}}(\pi(x)),$$
which by the inductive hypothesis implies that $\pi(x_n)\rightarrow\pi(x)$. Consequently, there exists $z_n\in \fz(\fn)$ such that $x_n*z_n\rightarrow x$. Since
$$f_{\varphi}(x*z)=f_{\varphi}(x)*f_{\varphi}(z), \hspace{.5cm}\forall x\in \fn, z\in \fz(\fn),$$
we conclude that
$$f_{\varphi_n|_{\fz(\fn)}}(z_n)=f_{\varphi_n}(z_n)=(f_{\varphi_n}(x_n))^{-1}*f_{\varphi_n}(x_n*z_n)\rightarrow (f_{\varphi}(x))^{-1}*f_{\varphi}(x)=0=f_{\varphi|_{\fz(\fn)}}(0),$$
which, by the abelian case, implies $z_n\rightarrow 0$ and hence $x_n\rightarrow x$, as stated.
\end{proof}

\begin{corollary}
    Let $Z$ be a topological space and let $\varphi:Z\rightarrow\mathrm{Aut}(\fn)$ and $f:Z\rightarrow \fn$ be continuous curves. If $\det (1-\varphi(z))\neq 0$ for any $z\in Z$, the curve
    $$z\in Z\mapsto f_{\varphi(z)}^{-1}(f(z))\in \fn,$$
    is continuous.
\end{corollary}

\subsection{Control-affine systems}

Let $M$ be a finite-dimensional smooth manifold and consider a compact and convex subset $\Omega \subset \mathbb{R}^{m}$ satisfying $0\in \inner\Omega$. A {\bf control-affine system} on $M$ is the (parametrized) family
of ordinary differential equations
\begin{flalign*}
\label{CAS}
&&\,
\dot{x}=f_{0}(x)+\sum_{j=1}^{m} u_{j}f_{j}(x),\hspace{.5cm} u=(u_1, \ldots, u_m)\in\UC
&&\hspace{-1cm}\left(\Sigma_{M}\right),
\end{flalign*}
where $f_{0},f_{1},\ldots ,f_{m}$ are smooth vector fields defined on $M$, and the $\UC$ is the set of the {\bf control functions} given by
$$\mathcal{U}:=\{u:\mathbb{R}\mapsto\mathbb{R}^m, \mbox{ locally integrable with }u(t)\in \Omega \mbox{ for all } t\in\mathbb{R}\}.$$
By our assumptions on the {\bf control range} $\Omega$, the set of $\UC$ is compact, locally convex, and metrizable in the weak*-topology of $L^{\infty}(\R, \R^m)=(L^1(\R, \R^m))^*$.

%We also assume w.l.o.g. $m\leq \dim M$ and that the set $\{f_1, \ldots, f_m\}$ is linearly independent in the set of the smooth vector fields on $M$.

For a given $x\in M$ and $u\in \mathcal{U}$, the solution of $
\Sigma _{M}$ is the unique absolutely continuous curve $t\mapsto \phi (t,x, u)
$ on $M$ satisfying $\phi (0, x,  u)=x$. 
We denote by 
\begin{align*}
\mathcal{O}_{S}^{+}(x)& :=\left\{ \phi (S, x, u)\ :\ u\in \mathcal{%
U}\right\} , \\
\mathcal{O}_{\leq S }^{+}(x)& :=\bigcup_{s \in
[0, s]}\mathcal{O}_{s}^{+}(x)\mbox{\quad and\quad }\mathcal{O}%
^{+}(x):=\bigcup_{\tau \text{ }>\text{ }0}\mathcal{O}_{\tau }^{+}(x).
\end{align*}%
the set of points {\bf reachable from $x\in M$ at time $S > 0$}, the set of points  {\bf reachable from $x$ up to time $S $}, and the {\bf positive orbit of $x$}, respectively. Analogously we denote by $\mathcal{O}_{S}^{-}(x), \mathcal{O}_{\leq S}^{-}(x)$ and $\mathcal{O}^{-}(x)$ the set of points {\bf controllable to $x$ in time $\tau>0$}, the set of points {\bf controllable to $x$ up to time $S$}, and the {\bf negative orbit of $x$}, respectively.

The control-affine system $\Sigma_M$ is said to be {\bf locally accessible} if for all $x\in M$ and $S >0$, 
the sets $\mathcal{O}_{\leq S }^{+}(x)$ and $\mathcal{O}_{\leq
S}^{-}(x)$ have nonempty interiors. Analogously, $\Sigma_M$ is said to be  {\bf strongly accessible} if for each $x\in M$, there is some $S>0$ such that $\inner \mathcal{O}^+_S(x)\neq 0$.

The previous properties can be obtained by algebraic means by considering the Lie algebra $\mathcal{L}$ generated by the vectors $f_u:=f_{0}+\sum_{j=1}^{m} u_{j}f_{j}, u\in\Omega$ and $\LC_0$ the ideal in $\LC$ generated by the differences $f_u-f_v, u, v\in\Omega$.  Then, if the system $\Sigma_M$ satisfies the {\bf Lie algebra rank condition (LARC)}, that is, if $\mathcal{L}(x)=T_{x}M$ for all $x\in M$, then it is locally accessible. For real analytic systems, the two concepts are in fact equivalents \cite[Theorem 12]{Sontag}. On the other hand, the system is strongly accessible if and only if $\mathcal{L}_0(x)=T_{x}M$ (see \cite[Proposition 5.6]{Ka2}). 

The next technical lemma will assure the existence of a control function in $\UC$ with a nice property. Such property will come in handy ahead, in order to assure the existence of a control set with a nonempty interior ahead.

\begin{lemma}\label{universalcontrol}
If the control-affine system $\Sigma_M$ is real analytic and satisfies the LARC, then it admits a control function $u_0\in\UC$ satisfying
\begin{equation}
 \label{inner}
\phi(S, x, u_0)\in\inner\OC^+(x), \hspace{.5cm}\forall x\in M, S>0.
\end{equation}
\end{lemma}
\begin{proof} For any $\alpha>1$, let us consider the control-affine systems on $M$ given by
\begin{flalign*}
\label{CAS}
&&\,
\dot{y}=f^{\alpha}(x, (u, v)):=v\left[f_0(y)+\sum_{i=j}^{m}u_jf_j(y)\right],\hspace{.5cm} (u, v)\in \mathcal{U}\times\mathcal{V}^{\alpha}
&&\hspace{-1cm}\left(\Sigma^{\alpha}_{M}\right),
\end{flalign*}
where $\mathcal{V}^{\alpha}:=\left\{v:\mathbb{R}\mapsto \mathbb{R},v(S)\in \left(\frac{1}{\alpha},\alpha\right),S\in\mathbb{R}\right\}.$
For each $\alpha>1$, it holds that $\LC\subset \LC_0^{\alpha}$. In fact, for the $v_1=v_2=v\in\left(\frac{1}{\alpha}, \alpha\right)$ and $u_1, u_2\in\Omega$, we get that
$$f^{\alpha}_{(v_1, u_1)}-f^{\alpha}_{(v_2, u_2)}=v(f_{u_1}-f_{u_2})\hspace{.5cm}\implies\hspace{.5cm}\LC_0\subset\LC_0^{\alpha}.$$
On the other hand, by considering $v_1=v_2=v\in\left(\frac{1}{\alpha}, \alpha\right)$ with $v_1\neq v_2$ and $u_1=u_2=0\in\Omega$, we get
$$f^{\alpha}_{(v_1, u_1)}-f^{\alpha}_{(v_2, u_2)}=(v_1-v_2)f_0\hspace{.5cm}\implies\hspace{.5cm}f_0\in\LC_0^{\alpha}.$$
Since, in the control-affine case, $\LC$ coincides with the Lie algebra generated by $f_0, f_1, \ldots, f_m$ and $\LC_0$ with the ideal generated by $f_1, \ldots, f_m$, the previous calculations imply that $\LC\subset \LC_0^{\alpha}$. Therefore, if $\Sigma_M$ satisfies the LARC, then
$$\forall x\in M, \hspace{.5cm}T_xM=\LC(x)\subset \LC_0^{\alpha}\subset T_xM,$$
which, by our previous discussion, implies that $\Sigma_M^{\alpha}$ is strongly accessible. Since these systems are also real analytic, Theorem A.4.15 of \cite{FCWK} assures the existence of a control function $(u^{\alpha}, v^{\alpha})\in\UC\times\mathcal{V}^{\alpha}$, satisfying
$$\phi^{\alpha}(S, x, u^{\alpha}, v^{\alpha}) \in \inner \mathcal{O}_S^{\alpha,+}(x), \hspace{.5cm}\forall x\in M, S>0,$$
where $\mathcal{O}_S^{\alpha,+}(x)$ is the positive orbit of $x$ for the system $\Sigma_M^{\alpha}$. However, by Lemma 4.5.18 of \cite{FCWK} and its proof, for any $x\in M$ and $S>0$, we get that
$$\OC^{\alpha, +}_{\leq S}(x)\subset\OC^+_{\leq S}(x),$$
and 
    \begin{equation*}
        \phi(\tau^{-1}(S),x,u)=\phi^{\alpha}(S,x,u^{\alpha},v^{\alpha})\in \inner\mathcal{O}^{\alpha,+}_S(x)\subset\inner\mathcal{O}^{\alpha,+}_{\leq S}(x)\subset\inner\mathcal{O}^+_{\leq \alpha S}(x)\subset\inner\mathcal{O}^+(x).
    \end{equation*}
where $\tau: \mathbb{R}^+ \mapsto \mathbb{R}^+$ is the bijective function defined as
$$ \tau(S):=\int_0^S v(t)dt \hspace{.5cm}\mbox{ and }\hspace{.5cm}u(S):=u^{\alpha}(\tau(S)),$$
concluding the proof.
\end{proof}

\bigskip

Next we introduce the concept of control sets encountered \cite[Definition 3.1.2]{FCWK}. 

\begin{definition}
\label{Control}
A set $D\subset M$ is a control set of $\Sigma_{M}$ if it is
maximal, w.r.t. set inclusion, with the following properties:
\begin{enumerate}
\item $\forall x\in D$, there exists a control $u\in \mathcal{U}$
such that $\phi \left(\mathbb{R}^{+},x, u\right) \subset D$;
\item $\forall x\in D$, it holds that $D\subset \overline{\mathcal{O}^{+}(x)}$.
\end{enumerate}
\end{definition}

The control sets of a system play an important role since
several dynamical entities of the system are contained in these sets, such as equilibrium points, recurrent
points, periodic and bounded orbits, etc (see \cite[Chapter 4]{FCWK}). In particular, exact controllability
holds in the interior of such control sets, that is, 
$$\forall x, y\in\inner D, \hspace{.5cm}\exists \;S>0, u\in\UC; \hspace{.5cm}y=\phi(S, x, u).$$

\begin{remark}
\label{interiorControl}
By \cite[Proposition 3.2.5]{FCWK}, any subset $D_0\subset M$ satisfying properties 1. and 2. in Definition \ref{Control} is contained in a control set. In particular, if $x\in\inner\OC^+(x)$, the set $D_0=\inner\OC^+(x)\cap \inner\OC^-(x)$ is a nonempty open set satisfying 1. and 2. Therefore, there exists a control set $D_x$ satisfying $x\in\inner D_x$. This fact will help us to assure the existence of control sets ahead.
\end{remark}

The next lemma shows that the continuous image of a connected topological space can be in the interior of, at most, one control set.

\begin{lemma}
\label{connected}
Let $X$ be a connected topological space and $\gamma:X\rightarrow M$ be a continuous map. If for any $x\in X$ there exists a control set $D_x$, satisfying $\gamma(x)\in \inner D_x$, then
$D_x=D_y$ for all $x, y\in X$.
\end{lemma}

\begin{proof} Since, for any $x\in X$, it holds that $\gamma(x)\in\inner D_x$, the continuity of $\gamma$ assures the existence of a neighborhood $V$ of $x$ such that $\gamma(V)\subset\inner D_x$, which implies that $D_y=D_x$ if $y\in V$. In particular, for any $x_0\in X$, the set
$X_0:=\{x\in X; D_x=D_{x_0}\}$ is open and closed, which by the connectedness of $X$ implies the result.
\end{proof}
\bigskip

In what follows, we introduce the concept of conjugation between control-affine systems and analyze their relationship with control sets.

Let $N$ be another smooth manifold and
\begin{flalign*}
\label{CAS}
&&\,
\dot{y}(t)=g_{0}(y(t))+\sum_{j=1}^{m}u_{j}(t)g_{j}(y(t)),,\hspace{.5cm} u=(u_1, \ldots, u_m)\in\UC
&&\hspace{-1cm}\left(\Sigma_N\right),
\end{flalign*}
a control-affine system on $N$.

\begin{definition}\label{defi:conjugado}
If $\psi :M\rightarrow N$ is a smooth map, we say
that $\Sigma _{M}$ and $\Sigma _{N}$ are $\psi $-conjugated if their
respective vector fields are $\psi $-conjugated, that is,
\begin{equation*}
\psi _{\ast }\circ f_{j}=g_{j}\circ \psi
\end{equation*}
for each $j\in \{0,1,\ldots ,m\}$. If such a $\psi $ exists, we say that $\Sigma _{M}$ and $\Sigma _{N}$ are conjugated. In particular, if $\psi $ is a diffeomorphism, then $\Sigma _{M}$ and $\Sigma _{N}$ are called
equivalent systems.
\end{definition}

It is straightforward to see that controllability, topological
properties of control sets, and positive (or negative) orbits are related in conjugated systems. For instance, if $\Sigma_M$ satisfies property (\ref{inner}), then certainly $\Sigma_N$ also satisfies it. This and other properties of conjugated systems will be heavily used to prove our main results in the next sections. In the same direction, the next result, whose proof can be found, for instance, in \cite[Proposition 3.10]{DS1}, relates control sets between conjugated systems.

%{\color{blue}

%Alem disso, aqui assumimos que $U$ e convexo, o que pela prop 3.2.29 do livro do Colonius-Kliemann nao altera os control sets!! Isso implica, em particular, que $\mathcal{U}$ tambem e convexo. Tambem, todo ponto $u_0\in\UC$ admite vizinhancas convexas neste caso, dadas por
%$$V_{u_0}:=\left\{u\in L^{\infty}(\R, \R^m), \left|\int_{\R}\langle u_0(t)-u(t), 
% x_j(t)\rangle dt\right|<\varepsilon, u(t)\in U\right\},$$
%e portanto, aberto e conexo implica conexo por caminhos!!

%}

\begin{proposition}
\label{conjugation}
    Let $\Sigma_M$ and $\Sigma_{N}$ be $\psi$-conjugated control-affine systems. It holds:
\begin{enumerate}
    \item[1.] If $D_M$ is a control set of $\Sigma_M$, there exists a control set $D_N$ of $\Sigma_{N}$ such that $\psi(D_M)\subset D_N$;
    \item[2.] If for some $y_0\in \inner D_N$ it holds that $\psi^{-1}(y_0)\subset\inner D_M$, then $D_M=\psi^{-1}(D_N)$.
\end{enumerate}
    
\end{proposition}

\section{Linear control systems and their control sets}

In this section we introduce the class of linear control systems on Lie groups. These systems appeared first as a natural generalization of their counterparts in Euclidean spaces (see \cite{Tirao}). Recently, however, their importance was highlighted by the Equivalence Theorem of \cite[Theorem 5]{JPh1}, which assures that any control-affine system on a connected manifold that is transitive and generates a finite-dimensional Lie algebra is equivalent to a linear control system on a Lie group or on a homogeneous space.

Let us, as previously, consider $G$ to be a connected Lie group with Lie algebra $\fg$ identified with the set of right-invariant vector fields. A {\bf linear control system (LCS)} on $G$ is defined by the family of ODEs,
\begin{flalign*}
&&\dot{g}(t)=\mathcal{X}(g)+\sum_{j=1}^mu_jY_j(g), \hspace{.5cm} u=(u_1, \ldots, u_m)\in\UC, 
&&\hspace{-1cm}\left(\Sigma_G\right)
\end{flalign*}
where $\XC$ is a linear vector field and $Y_0, Y_1 \ldots, Y_m$ are right-invariant ones. For any $g\in G$ and $u\in \UC$, the solution $S\mapsto \phi(S, g, u)$ of $\Sigma_G$ is defined on the whole real line $\R$ and satisfies
$$\phi_{S, u}\circ R_g=R_{\varphi_S(g)}\circ\phi_{S, u},$$
where, for simplicity, $\phi_{S, u}(g):=\phi(S, g, u)$. Since linear and invariant vector fields are analytics, if $\Sigma_G$ satisfies the LARC, then it admits a control function $u_0\in\UC$ satisfying property (\ref{inner}). We will use this fact ahead.

Another fact that will be important ahead is conjugations given by canonical projections. Let $H\subset G$ be a Lie subgroup and consider the homogeneous space $G/H$. According to \cite[Proposition 4]{JPh1}, a linear vector field projects to a well-defined vector field on $G/H$ if and only if $H$ is $\varphi$-invariant. Therefore, for a $\varphi$-invariant subgroup $H$, a linear control system $\Sigma_G$ induces a $\pi$-conjugated control-affine system $\Sigma_{G/H}$ on the homogeneous space $G/H$, where $\pi:G\rightarrow G/H$ is the canonical projection. 

%The previous formula, implies the following properties for the positive orbits of $\Sigma_G$ (see \cite[Proposition 2]{JPh}).

%\begin{itemize}
 %   \item[1.] $\OC^+_{\leq S}(e)=\OC^+_S(e)$;
  %  \item[2.] $\OC^+_{S_1}(e)\subset \OC^+_{S_2}(e)$ for all $0<S_1<S_2$;
   % \item[3.] $\OC^+_S(g)=\OC^+_S(e)\varphi_S(g)$, for all $g\in G, S>0$;
  %  \item[4.] $\OC^+_{S_1+S_2}(e)=\OC^+_{S_1}(e)\varphi_{S_1}\left(\OC^+_{S_2}(e)\right)=\OC^+_{S_2}(e)\varphi_{S_2}\left(\OC^+_{S_1}(e)\right)$ for all $S_1, S_2>0$
%\end{itemize}

Denote by $G_0$ the generalized kernel of $\XC$ and by $N$ the nilradical of $G$. The generalized kernel of the restriction $\XC|_N$ is then $N_0=N\cap G_0$. Our main result states that if $\Sigma$ satisfies the LARC and $N_0$ is a compact subgroup, then $\Sigma$ admits a unique control set with a nonempty interior. Moreover, this control set contains $G_0$ in its closure. 

In order to prove the previous, we start by showing that we can get rid of any $\varphi$-invariant compact subgroup of $G$. This will help us to simplify our analysis to simply connected spaces.

\begin{proposition}
\label{compactfiber}
  Let $\Sigma_G$ be a linear control system on $G$ satisfying the LARC and $H\subset G$ a $\varphi$-invariant compact subgroup. If the induced control-affine system $\Sigma_{G/H}$ on the homogeneous space $G/H$ admits a control set $D$ with a nonempty interior, then $\pi^{-1}(D)$ is a control set of $\Sigma_G$ with a nonempty interior, where $\pi:G\rightarrow G/H$ is the canonical projection.
\end{proposition}

\begin{proof}
Since the canonical projection $\pi: G\rightarrow G/H$ conjugates the systems, it holds that
$$\forall g\in G, S>0, \hspace{.5cm}\pi(\OC_{S}^+(g))=\widehat{\OC}_{S}^+(\pi(g)).$$
The previous, together with the continuity of $\pi$ and the fact that $\inner\OC_{S}^+(g)$ is dense in $\OC^+_{S}(g)$ for any $g\in G$ and $S>0$ (see \cite[Chapter3, Theorem 2]{Jurd}), implies that
$$\pi(g)\in\pi\left(\overline{\inner\OC_{S}^+(g)}\right)\subset\overline{\pi(\inner\OC^+_{S}(g))}.$$
Let then $x\in\inner D$ and write it as $x=\pi(g)$. By exact controllability, there exists $S_0>0$ such that $x\in\widehat{\OC}_{S_0}^+(x)=\pi(\OC^+(g))$. Therefore,
$$\inner D\ni x\in \pi(\OC_{S_0}^+(g))\subset \pi(\overline{\OC_{S_0}^+(g)})=\pi\left(\overline{\inner\OC_{S_0}^+(g)}\right)\subset\overline{\pi(\inner\OC^+_{S_0}(g))},$$
implying the existence of $g'\in\inner\OC^+_{S_0}(g)$ such that $\pi(g')\in\inner D$. By exact controllability, there exists $u_1\in\UC$ and $S_1>0$ such that
$$\pi(\phi(S_1, g', u_1))=\widehat{\phi}(S_1, \pi(g'), u_1)=\pi(g)\hspace{.5cm}\implies \hspace{.5cm} \emptyset\neq gH\cap \phi_{S_1, u_1}(\inner\OC^+_{S_0}(g))\subset xH\cap\inner\OC^+_{S_1+S_0}(g)$$

The previous shows that, for any $g\in \pi^{-1}(\inner D)$, there exists $S>0$ such that
 $$g^{-1}\inner\OC_S^+(g)\cap H\neq\emptyset.$$
 Since $H$ is a compact subgroup, $\varphi_S|_{H}=C_{\rme^{SX}}$ for some $X\in \fh$. Moreover, the fact that $\rme^{S X}\in \fix(\varphi|_H)\subset H$ implies that
 $$R_{\rme^{-S X}}\left(x^{-1}\inner\OC_{S}^+(x)\cap H\right),$$
 is an open subset of $H$. As a consequence, there exists $g_1\in x^{-1}\inner\OC_{S}^+(x)\cap H$ such that $g_1\rme^{-S X}$ has finite order. Let $u_1\in\UC$ such that $\phi(S, x, u_1)=xg_1$ and extend it $S$-periodically. Inductively, one easily gets that
  $$\forall n\in\N, \hspace{.5cm}\phi(nS, x, u_1)=x\varphi_{(n-1)S}(g_1)\cdots \varphi_{S}(g_1)g_1=x\rme^{(n-1)S X}(g_1\rme^{-S X})^{n}\rme^{S X}.$$
  Therefore, if $k_1\in \N$ is such that $(g_1\rme^{-S X})^{k_1}=e$, it holds that
$$\phi(k_1S, x, u_1)=x\rme^{k_1S X}\hspace{.5cm}\implies\hspace{.5cm}x^{-1}\inner\OC^+_{k_1S}(x)\cap\fix(\varphi|_H)\neq\emptyset.$$
Since $\fix(\varphi|_H)$ is a compact subgroup and $x^{-1}\inner\OC^+_{k_1S}(x)\cap\fix(\varphi|_H)$ is a nonempty open subset of $\fix(\varphi|_H)$, there exist $g_2\in \fix(\varphi|_H)$, $u_2 \in\UC$, and $k_2\in\N$ such that
$$\phi(k_1S, x, u_2)=xg_2\in\inner\OC^+_{k_1 S}(x)\hspace{.5cm}\mbox{ and }\hspace{.5cm}g_2^{k_2}=e.$$
As in the previous case, one can show inductively that
$$\forall m\in\N, \hspace{.5cm}\phi(mk_1 S, x, u_2)=xg^m_2\hspace{.5cm}\implies\hspace{.5cm}\phi(k_2k_1S, x, u_2)=x\in\inner\OC^+_{k_2k_1S}(x).$$
In particular, $x\in\inner\OC^+(x)$ implies the existence of a control set $C_x$ with $x\in\inner C_x$ for any $x\in\pi^{-1}(\inner D)$. Since the fiber $\pi^{-1}(\pi(x))=xH$ is path-connected, we concluded by Lemma \ref{connected} that
$$\forall y\in \pi^{-1}(\pi(x)), \hspace{.5cm} C_x=C_y:=C.$$
Since $x\in\inner D$ and $\pi^{-1}(\pi(x))\subset\inner C$, Proposition \ref{conjugation} implies that $C=\pi^{-1}(D)$ is a control set of $\Sigma_G$ as stated.
\end{proof}

\subsection{Uniqueness}

In this section we prove that, if $N_0$ is a compact subgroup, the LCS $\Sigma$ admits, at most, one control set with a nonempty interior.

Before stating and proving such a result, let us make a comment that will help us in the proof. Let $H\subset G$ be a connected normal subgroup with Lie algebra $\fh$. If $H$ is $\varphi$-invariant, the LCS $\Sigma_G$ induces a $\pi$-conjugated LCS $\Sigma_{\widehat{G}}$ on the homogeneous space $\widehat{G}:=G/H$. If $\fh\subset\fn$, then the nilradical $\widehat{\fn}$ of $\widehat{\fg}=\fg/\fh$ satisfies $(d\pi)_e^{-1}(\widehat{\fn})=\fn$. Moreover, since $\pi$ conjugates the system, it holds that
$$(d\pi)_e\circ\DC=\widehat{\DC}\circ(d\pi)_e,$$
where $\DC$ and $\widehat{\DC}$ are the derivations of the systems $\Sigma_G$ and $\Sigma_{\widehat{G}}$, respectively. The previous relation implies directly that
\begin{equation}
\label{nilradicalinduzido}
 \widehat{\fn}_0=(d\pi)_e\left(\{X\in\fn; \DC^nX\in\fh, \mbox{ for some }n\in\N\}\right).
\end{equation}

We can now state and prove the main result of this section.

\begin{theorem}
\label{uniqueness}
Let $\Sigma_G$ be a linear control system on a solvable Lie group $G$, and denote by $N$ the nilradical of $G$. If $N_0=N\cap G_0$ is compact, then $\Sigma_G$ admits (at most) one control set with a nonempty interior.
\end{theorem}

\begin{proof}
The proof is done by induction on the dimension of $G$. If $G$ is nilpotent, the result is already true by \cite[Theorem 3.4]{DS1}. In particular, the result is true when $\dim G=1$.

Assume the result to holds true on any solvable Lie group of dimension smaller than $n$, and let $\Sigma_G$ be a LCS on a $n$-dimensional solvable Lie group $G$. Since compact subgroups of nilpotent groups live in the center, by our hypothesis $N_0\subset Z(N)$. Therefore, by using that $\fg=\fn+\fg_0$ (see Lemma \ref{conj}), allow us to conclude that $N_0$ is a normal subgroup of $G$.

Let us first assume that $\dim N_0\neq 0$. Since $N_0$ is a $\varphi$-invariant, closed, connected, and normal subgroup of $G$, it induces on $G^1:=G/N_0$ a LCS $\Sigma_{G^1}$. Moreover, since
$$\{X\in\fn; \DC^nX\in\fn_0, \mbox{ for some }n\in\N\}=\fn_0,$$
relation (\ref{nilradicalinduzido}), implies that $\fn^1_0=(d\pi_1)_e(\fn_0)=\{0\}$,
 implying that $N^1_0=\{e\}$ is a compact subgroup. Since $G^1$ is a solvable Lie group with $\dim G^1<n$, the LCS $\Sigma_{G^1}$ admits at most one control set with a nonempty interior.

Now, if $C$ is a control set with a nonempty interior of $\Sigma_G$, by Proposition \ref{conjugation}, there exists a control set $D$ of $\Sigma_{G^1}$ with a nonempty interior. Since $G^1$ is in the hypothesis of the theorem, we conclude that $D$ is unique. On the other hand, the fact that $N_0$ is a compact subgroup implies by Proposition \ref{compactfiber} that  $\pi^{-1}(D)$ is a control set with a nonempty interior of $\Sigma_G$. Since $C\subset\pi^{-1}(D)$, we must have $C=\pi^{-1}(D)$, showing that $C$ has to be the only control set with a nonempty interior of $\Sigma_G$.

Let us now assume that $\dim N_0=0$. In this case, $\DC$ restricted to $\fn$ is invertible, and $\fg=\fg_0\oplus\fn$. Let us consider the decomposition of $\fz(\fn)=\fz(\fn)^+\oplus \fz(\fn)^0\oplus \fz(\fn)^-$, where
$\fz(\fn)^+, \fz(\fn)^0$, and $\fz(\fn)^-$ are the sums of generalized eigenspaces of $\DC|_{\fz(\fn)}$ with positive, zero, and negative real parts, respectively. Since, by standard properties of the eigenspaces of derivations (see \cite[Proposition 3.1]{SM1}), it holds that
$$[\fg^0, \fz(\fn)^+]\subset \fz(\fn)^+, \hspace{.5cm} [\fg^0, \fz(\fn)^0]\subset \fz(\fn)^0\hspace{.5cm}\mbox{ and }\hspace{.5cm}[\fg^0, \fz(\fn)^-]\subset \fz(\fn)^-,$$
 we have that $\fz(\fn)^+, \fz(\fn)^0$, and $\fz(\fn)^-$ are in fact ideals of $\fg$. We have the following possibilities:

\begin{itemize}
\item[1.] $\fz(\fn)^+\neq\{0\}$ or $\fz(\fn)^-\neq\{0\}$:  Let us consider the case $\fz(\fn)^-\neq\{0\}$, since the other possibility is analogous. The connected subgroup $Z(N)^-$ associated with $\fz(\fn)^-$ is a $\varphi$-invariant, closed, connected, and normal subgroup of $G$. Therefore, the LCS $\Sigma_G$ induces a LCS $\Sigma_{G^-}$ on the solvable Lie group $G^-:=G/Z(N)^-$. Since the restriction of $\DC$ to $\fn$ is invertible, we get that
$$X\in\fn\hspace{.5cm}\mbox{ and }\hspace{.5cm}\DC^nX\in\fz(\fn)^-, \mbox{ for some }n\in\N\hspace{.5cm}\iff\hspace{.5cm} X\in \fz(\fn)^-,$$
which by relation (\ref{nilradicalinduzido}) implies that $\fn^-_0=\{0\}$. Since $G^-$ is solvable, $\dim G^-<n$, and $N^-_0=\{e\}$, we conclude that the LCS $\Sigma_{G^-}$ admits at most one control set with a nonempty interior.

 Let $C_1, C_2$ be control sets with nonempty interior of $\Sigma_G$ and denote by $D$ the unique control set of $\Sigma_{G^2}$ containing the projections $\pi(C_1)$ and $\pi(C_2)$. Since $\pi(\inner C_i)\subset\inner D$ for $i=1, 2$, for any $x\in\inner C_1$ and $y\in\inner C_2$, there exists $S_1>0$ and $u_1\in\UC$ such that
$$\phi^2(S_1, \pi(x), u_1)=\pi(y)\hspace{.5cm}\mbox{ or equivalently } \hspace{.5cm}\phi(S_1, x, u_1)=yh, \mbox{ for some }h\in Z(N)^-.$$
On the other hand, the fact that exact controllability holds in $\inner C_2$ implies the existence of $S_2>0$ and $u_2\in\UC$ such that
$$\forall n\in\N, \hspace{.5cm}\phi(nS_2, y, u_2)=y\hspace{.5cm}\implies \hspace{.5cm}\phi(nS_2, \phi(S_1, x, u_1), u_2)=y\varphi_{nS_2}(h).$$
However, the fact that $\DC$ restricts to $\fz(\fn)^-$ has only eigenvalues with negative real parts, implies that
$$\varphi_{nS_2}(h)\rightarrow e, \hspace{.5cm}\mbox{ as }\hspace{.5cm}n\rightarrow+\infty,$$
and hence, $\phi(n_0S_2, \phi(S_1, x, u_1), u_2)\in\inner C_2$ for some $n_0\in\N$. Furthermore, for some $S_3>0$ and $u_3\in\UC$, it holds that
$$\phi(S_3, \phi(n_0S_2, \phi(S_1, x, u_1), u_2), u_3)=y,$$
showing the existence of a trajectory starting in $x\in\inner C_1$ and ending in $y\in\inner C_2$. By interchanging the roles of $x$ and $y$, we obtain a periodic orbit passing through $x$ and $y$, which is only possible when $C_1=C_2$, proving the uniqueness in this case.

\item[2.] $\fz(\fn)^+=\fz(\fn)^-=\{0\}$:
Since the restriction of $\DC$ to $\fn$ is invertible, we get that $\det \DC|_{\fz(\fn)}\neq 0$. Therefore, $\fz(\fn)=\fz(\fn)^0$ implies that $\varphi_S$ restricted to $Z(N)$ is a flow of isometries (see \cite[Theorem 2.4]{DSAyPH}). In particular, for any $h\in Z(N)$, the orbit $\{\varphi_S(h), S\in\R\}$ is bounded.

Let us assume that $\Sigma_G$ admits a control set $C$ with a nonempty interior and consider $x\in\inner C$. Let $S_0>0$ and $u\in\UC$ such that $\phi(nS_0, x, u)=x$ for any $n\in\N$. Then, for any $h\in Z(N)$, it holds that
$$\phi(nS_0, xh, u)=x\varphi_{nS_0}(h)\hspace{.5cm}\implies\hspace{.5cm}\{\phi(S, xh, u), S\geq 0\}\hspace{.5cm}\mbox{ is bounded}.$$
Since we are assuming that $\Sigma_G$ satisfies the LARC, Corollary 4.5.9 of \cite{FCWK} implies the existence of a control set $C_h$ such that
$$\{y\in G; \exists S_k\rightarrow+\infty; \phi(S_k, xh, u)\rightarrow y\}\subset\inner C_h.$$
On the other hand, since $\varphi_S$ restricted to $Z(N)$ is a flow of isometries, the sequence $\{\varphi_{nS_0}(h)\}_{n\in\N}$ admits a subsequence $\{\varphi_{n_kS_0}(h)\}_{k\in\N}$ that converges to $h$ (see Remark \ref{elliptical}), and hence,
implying that $xh\in\inner C_h$. Since $Z(N)$ is path-connected, we get by Lemma \ref{connected} that $C_h=C$ for any $h\in Z(N)$, implying that $xH\subset \inner C$.

Now, the homogeneous space $G^0=G/Z(N)$ is solvable, and its dimension is smaller than $n$. Moreover, the invertibility of the restriction $\DC|_{\fn}$ implies, as in the previous case, that $\fn^0_0=\{0\}$. Therefore, the induced system $\Sigma_{G^0}$ satisfies the inductive hypothesis, and hence, the control set $D$ of $\Sigma_{G^0}$ that contains $\pi(C)$ is unique. Moreover, by the previous calculations, if
$$x\in\inner C\hspace{.5cm}\implies\hspace{.5cm}\pi(x)\in\pi(\inner C)\subset\inner D\hspace{.5cm}\mbox{ and }\hspace{.5cm}\pi^{-1}(\pi(x))=xH\subset\inner C,$$
which by Proposition \ref{conjugation} gives us that $C=\pi^{-1}(D)$, implying the uniqueness of $C$ and concluding the proof.
\end{itemize}
\end{proof}

\section{Existence}

 In order to prove the existence of a control set for LCSs on solvable groups, we first specialize our analysis of the existence to some particular classes of control-affine systems. In addition to their own natural importance, these systems will appear naturally when conjugating a linear control system on a solvable Lie group by the homomorphism constructed on Lemma \ref{conj}.

\subsection{Control sets with a nonempty interior for a particular class of control-affine systems on nilpotent Lie groups}

Let $\fn$ be a nilpotent Lie algebra and identify it with the connected, simply connected Lie group $(\fn, *)$, where
the product is given by the BCH formula.
Such an identification allows us to work with the vectorial structure of $\fn$ and the group product on the same space. In order to make a distinction between the elements in the algebra and the group, we use capital letters $X, Y, Z, \ldots$ for the elements in $\fn$ seen as vectors and small-sized letters $x, y, z, \ldots $, for the elements of $\fn$ seen as group elements.

Under such identification, the group $\mathrm{Aut}(\fn)$ is, at the same time, the group of automorphisms of the algebra and of the group. As a consequence, the set of linear vector fields coincides with $\mathrm{Der}(\fn)$ (see \cite[Theorem 2.2]{Tirao} for the general case). Also, due to the vectorial structure of the subjacent manifold, a right-invariant vector field can be explicitly calculated as
$$Z(x)=(dR_x)_0Z=\frac{d}{ds}\Bigl|_{s=0} tZ*x$$

$$=\frac{d}{ds}\Bigl|_{s=0}\left(sZ+x+\frac{1}{2}[sZ, x]+\frac{1}{12}\left([sZ, [sZ, x]]+[x, [x, sZ]]\right)+\cdots\right)=\sum_{j=0}^{k-1}c_j\ad(x)^jZ,$$
where the coefficients $c_j$ comes from the BCH formula and $k\in\N$ is the degree of nilpotence of $\fn$.

Let us consider the control-affine system on $\fn$ given by
\begin{flalign*}
\label{affineN}
&&\dot{x}=\DC(u)x+\sum_{j=1}^mu_j(Z_j(x)), \hspace{.5cm}u\in\UC &&\hspace{-1cm}\left(\Sigma^A_{(\fn, *)}\right)
\end{flalign*}
where $\DC(u):=\DC_0-\sum_{j=1}^mu_j(t)\DC_j$, with $\DC_0, \DC_1, \ldots,\DC_m\in\mathrm{Der}(\fn)$ and $Z_1, \ldots, Z_m\in\fn$. Consider the associated system given by 
\begin{flalign*}
&&\dot{x}=\DC(u)x=\left(\DC_0-\sum_{j=1}^mu_j\DC_j\right)x,\hspace{.5cm}u\in\UC.  &&\hspace{-1cm}\left(\Sigma_{(\fn, *)}^B\right)
\end{flalign*}
 For any $x\in \fn$ and $u\in\UC$, we denote by 
$$t\mapsto \varphi^B_{S, u}(x):=\varphi^B(S, x, u)\hspace{.5cm} \mbox{ and } \hspace{.5cm}t\mapsto \phi^A_{S, u}(x):=\phi^A(S, x, u),$$ the solutions curves of $\Sigma_{(\fn, *)}^B$ and $\Sigma_{(\fn, *)}^A$, respectively. These curves are defined for all $S\in \R$ and they satisfy a relation similar to the LCS (see \cite[Theorem 4.1]{DS2})
\begin{equation}
\label{solution}
\forall t\in\R, x\in \fn, u\in\UC, \hspace{.5cm}\phi^A_{S, u}\circ R_{x}=R_{\varphi^B_{S, u}(x)}\circ\phi^A_{S, u}.
\end{equation}
%In fact, for any $\DC\in\mathrm{Der}(\fn)$ it holds that
%$$\DC (x*y)=(dL_x)_y\DC y+(dR_y)_x\DC x.$$
%then  
%$$\frac{d}{dt} \phi^A_{t, u}(y)*\varphi^B_{t, u}(x)=(dL_{\phi^A_{t, u}(y)})_{\varphi^B_{t, u}(x)}\frac{d}{dt}\varphi^B_{t, u}(x)+(dR_{\varphi^B_{t, u}(x)})_{\phi^A_{t, u}(y)}\frac{d}{dt}\phi^A_{t, u}(y)$$
%$$=(dL_{\phi^A_{t, u}(y)})_{\varphi^B_{t, u}(x)}\DC(u(t))\varphi^B_{t, u}(x)+(dR_{\varphi^B_{t, u}(x)})_{\phi^A_{t, u}(y)}\left(\DC(u(t))\phi^A_{t, u}(y)+\sum_{j=1^m}u_j(t)Z_j(\phi^A_{t, u}(y))\right)$$
%$$=\left((dL_{\phi^A_{t, u}(y)})_{\varphi^B_{t, u}(x)}\DC(u(t))\varphi^B_{t, u}(x)+(dR_{\varphi^B_{t, u}(x)})_{\phi^A_{t, u}(y)}\DC(u(t))\phi^A_{t, u}(y)\right)+\sum_{j=1^m}u_j(t)Z_j\left(\phi^A_{t, u}(y)*\varphi^B_{t, u}(x)\right)$$
%$$\DC(u(t))\left(\phi^A_{t, u}(y)*\varphi^B_{t, u}(x)\right)+\sum_{j=1^m}u_j(t)Z_j\left(\phi^A_{t, u}(y)*\varphi^B_{t, u}(x)\right).$$
%Therefore, the curve $\alpha(t):=\phi^A_{t, u}(y)*\varphi^B_{t, u}(x)$ is the curve satisfying the differential equation $\Sigma^A$ and satisfying $\alpha(0)=x*y$ implying the assertion.

The property (\ref{solution}) for the solutions of $\Sigma_{(\fn, *)}^A$, together with Lemma \ref{universalcontrol}, implies the following result.

\begin{theorem}
\label{existence}
If the affine control system $\Sigma_{(\fn, *)}^A$ on $(\fn, *)$ satisfies the LARC and $\det \DC_0\neq 0$, then it admits control sets with a nonempty interior.
\end{theorem}

\begin{proof}
Since the control-affine system $\Sigma_{(\fn, *)}^A$ is real analytic and satisfies the LARC, Lemma \ref{universalcontrol} assures the existence of a control function $u_0\in\UC$ satisfying property (\ref{inner}), that is,
$$\phi^A(S, x, u_0)\in\inner\OC^{+, A}(x), \hspace{.5cm}\forall x\in\fn, S>0.$$
Now, the assumption that $\det\DC_0\neq 0$ assures the existence of $u\in\UC$ and $S>0$ such that
 $\det(\varphi^B_{S, u}-1)\neq 0$. Since,
$$t\rightarrow 0\hspace{.5cm}\implies\hspace{.5cm}\varphi^B_{t, u_0}\rightarrow 1,$$
there exists $\varepsilon>0$ such that
$$\forall t\in(0, \varepsilon), \hspace{.5cm}\det(\varphi^B_{S+t, u_{\varepsilon}}-1)=\det(\varphi^B_{S, u}\circ\varphi^B_{t, u_0}-1)\neq 0,$$
 where the control function $u_{\varepsilon}$
 \begin{equation}
\label{concatenation}
u_{\varepsilon}(t):=\left\{\begin{array}{ll}
 u_0(t) & t\in (0, \varepsilon) \\
 u(t) & t\in\R\setminus (0, \varepsilon)
\end{array}\right.,
        \end{equation}
is the concatenation of $u_0$ and $u$. By Lemma \ref{lemma}, for any $t\in(0, \varepsilon)$, the map $f_{\varphi^B_{S+t, u_{\varepsilon}}}$ is a diffeomorphism, and hence, there exists $x=x_t\in N$ such that
$$f_{\varphi^B_{S+t, u_{\varepsilon}}}(x)=\phi^A_{S+t, u_{\varepsilon}}(e)\hspace{.5cm}\implies\hspace{.5cm} x*\varphi_{S+t, u_{\varepsilon}}(x)^{-1}=\phi^A_{S+t, u_{\varepsilon}}(e)$$
$$\implies\hspace{.5cm}x=\phi^A_{S+t, u_{\varepsilon}}(e)*\varphi_{S+t, u_{\varepsilon}}(x)=\phi^A_{S+t, u_{\varepsilon}}(x).$$
 Since $\phi^A_{t, u_0}(x)\in\inner\OC^+(x)$ we get that
$$x=\phi^A_{S+t, u_{\varepsilon}}(x)=\phi^A_{S, u}\left(\phi^A_{t, u_0}(x)\right)\in \phi^A_{S, u}(\inner\OC^+(x))\subset\inner\OC^+(x),$$
which implies the existence of a control set $D_x$ with $x\in\inner D_x$ (see Remark \ref{interiorControl}), and proves the result.
\end{proof}
\bigskip

The previous result shows that the control-affine system $\Sigma_{(\fn, *)}^A$ admits a control set with a nonempty interior, as soon as it satisfies the LARC and the derivation $\DC_0$ is invertible. In what follows, we further explore these conditions in order to assure the existence of one such control set whose closure contains the identity element $0\in \fn$.

Assume as in Theorem \ref{existence} that $\det \DC_0\neq 0$ and let $S>0$ such that $\det(1-\rme^{S\DC_0})\neq 0$. Define the set
$$\UC_S:=\{u\in\UC; \;\det(1-\varphi^B_{S, u})\neq 0\}.$$
Since $\varphi^B_{S, \mathbf{0}}=\rme^{S \DC_0}$, the set $\UC_S$ is an open neighborhood of $\mathbf{0}$ in $\UC$, where $\mathbf{0}(t)=0$ for all $t\in\R$. Denote by $\UC_S^0$ the connected component of $\UC_S$ that contains $\mathbf{0}$ and note that, since $\UC_S^0$ is open and connected, it is in fact path connected.

\begin{proposition}
\label{dense}
There exists a control set $D$ of $\Sigma_{(\fn, *)}^A$ with a nonempty interior and such that
$$f_{\varphi_{S, u}^B}^{-1}(\phi^A_{S, u}(0))\in\overline{D}, \hspace{.5cm}\forall u\in\UC_S^0.$$
Moreover, the set
$$\left\{u\in\UC_S^0; f_{\varphi_{S, u}^B}^{-1}(\phi^A_{S, u}(0))\in\inner D\right\},$$
is dense in $\UC_S^0.$
\end{proposition}

\begin{proof} Let us previously consider $u_0\in\UC$ satisfying property (\ref{inner}) for $\Sigma_{(\fn, *)}^A$. For any $u\in\UC_S$ and $\varepsilon>0$, let us consider the control function $u_{\varepsilon}$ as defined in (\ref{concatenation}). Since $u_{\varepsilon}\rightarrow u$ as $\varepsilon\rightarrow 0$, we can assume w.l.o.g. that $u_{\varepsilon}\in\UC_S^0$. As in the proof of Theorem \ref{existence}, the points $f_{\varphi_{S+t, u_{\varepsilon}}^B}^{-1}(\phi^A_{S+t, u_{\varepsilon}}(0))$ belong to the interior of a control set $D_{u, t}$ for any $t\in(0, \varepsilon)$. Since the map
$$t\in(0, \varepsilon)\mapsto f_{\varphi_{S+t, u_{\varepsilon}}^B}^{-1}(\phi^A_{S+t, u_{\varepsilon}}(0))\in \fn,$$
is continuous, Proposition \ref{continverse} implies that
$D_{t, u}=:D_u$ for all $t\in (0, \varepsilon)$. Moreover,
$$t\rightarrow 0\hspace{.5cm}\implies\hspace{.5cm}f_{\varphi_{S+t, u_{\varepsilon}}^B}^{-1}(\phi^A_{S+t, u_{\varepsilon}}(0))\rightarrow f_{\varphi_{S, u}^B}^{-1}(\phi^A_{S, u}(0))\hspace{.5cm}\implies\hspace{.5cm} f_{\varphi_{S, u}^B}^{-1}(\phi^A_{S, u}(0))\in \overline{D_u}.$$

Let $\gamma:[0, 1]\rightarrow \UC_S^0$ be a continuous path satisfying $\gamma(0)=\mathbf{0}$ and $\gamma(1)=u$. For any $t\in[0, 1]$ there exists $\varepsilon(\tau)>0$ such that $\gamma_{\varepsilon(\tau)}(\tau)\in\UC_S^0$, where $\gamma_{\varepsilon(\tau)}(\tau)$ is also defined by the concatenation with $u_0$ as in (\ref{concatenation}). By continuity, $\gamma_{\varepsilon(\tau)}(\tau-\delta, \tau+\delta)\subset U_S^0$ for some $\delta>0$. Since $[0, 1]$ is compact, there exist $\tau_1, \ldots, \tau_n\in[0, 1]$ and $\delta_1, \ldots, \delta_n\in(0, +\infty)$ such that
$$[0, 1]\subset \bigcup_{i=1}^n(\tau_i-\delta_i, \tau_i+\delta_i)\hspace{.5cm}\mbox{ and }\hspace{.5cm} \gamma_{\varepsilon(\tau_i)}(\tau_i-\delta_i, \tau_i+\delta_i)\subset \UC_S^0.$$
Therefore, for all $\tau\in [0, 1]$ there exists $i\in\{1, \ldots, n\}$ such that
$$\forall t\in (0, \varepsilon(\tau_i)), \hspace{.5cm}\det(1-\varphi_{S, \gamma(\tau)}^B\circ\varphi_{t, u_0}^B)=\det\left(1-\varphi_{S, \gamma_{\varepsilon(\tau_i)}(\tau)}^B\right)\neq 0.$$
In particular, by considering $\varepsilon:=\min\{\varepsilon(\tau_1), \ldots, \varepsilon(\tau_n)\}$, we conclude that $\gamma_{\varepsilon}(\tau)\in\UC_S^0$ for all $\tau\in[0, 1]$. Since,
$$(t, \tau)\in(0, \varepsilon)\times [0, 1]\mapsto f_{\varphi_{S+t, \gamma_{\varepsilon}(\tau)}^B}^{-1}(\phi^A_{S+t, \gamma_{\varepsilon}(\tau)}(0))\in \fn,$$
is a continuous curve and $(0, \varepsilon)\times [0, 1]$ is connected, Lemma \ref{connected} implies that $D_u\equiv D$ showing that
$$f_{\varphi_{S, u}^B}^{-1}(\phi^A_{S, u}(0))\in\overline{D}, \hspace{.5cm}\forall u\in\UC_S^0.$$
On the other hand, for any $u\in\UC_S^0$, the functions $u_\varepsilon$ are, by construction, as close as we want to $u$. Since, 
$$f_{\varphi_{S, u_{\varepsilon}}^B}^{-1}(\phi^A_{S, u_{\varepsilon}}(0))\in\inner D,$$
we conclude that the set
$$\left\{u\in\UC_S^0; f_{\varphi_{S, u}^B}^{-1}(\phi^A_{S, u}(0))\in\inner D\right\},$$
is dense in $\UC_S^0$.
\end{proof}

\subsection{Systems on a Cartesian product}

Let $V$ be a finite-dimensional vector space and $\fn$ a nilpotent Lie algebra. In this section we consider the control-affine system on the Cartesian product $V\times \fn$ given by
\begin{flalign*}
\label{systemcartesian}
&&\, \left\{
\begin{array}{l}
\dot{v}=Av+\sum_{j=1}^mu_jb_j\\
\dot{x}=\DC(u)x+\sum_{j=1}^mu_jZ_j(x),
\end{array}
\right. &&\hspace{-1cm}\left(\Sigma_{V\times  \fn}\right),
\end{flalign*}
where the first equation is an LCS $\Sigma_V$ on the vectorial space $V$ and the second one $\Sigma_{(\fn, *)}^A$ is an affine control system on $(\fn, *)$, as in the previous section. If $\Sigma_{V\times \fn}$ satisfies the LARC, then both systems $\Sigma_V$ and $\Sigma^A_{(\fn, *)}$ also satisfy the LARC. Moreover, if $u_0\in\UC$ satisfies (\ref{inner}) for $\Sigma_{V\times\fn}$, it also satisfies the same property for both systems $\Sigma^A_{(\fn, *)}$ and $\Sigma_V$. By the results of the previous section, $\Sigma_{(\fn, *)}^A$ admits a control set with a nonempty interior $D$ such that $0\in\overline{D}$. Here we show that, if $\pi_2:V\times \fn\rightarrow \fn$ is the projection onto the second coordinate, then $\pi_2^{-1}(D)=V\times D$ is the only control set with a nonempty interior for the system $\Sigma_{V\times \fn}$.

Let us use the notation $\phi=(\phi_1, \phi_2)$ for the solution of $\Sigma_{V\times \fn}$. Since the first component of $\phi$ is the solution of a linear control system on the vectorial space $V$, it satisfies
$$\phi_1(S, v, u)=\rme^{SA}v+\phi_1(S, 0, u).$$
Therefore, for any $v\in\ker A$, we get that 
\begin{equation}
    \label{sum}
    \phi(S, (v, x), u)=(v, 0)+\phi(S, (0, x), u), \hspace{.5cm}\forall u\in\UC, S\in\R.
\end{equation}
 
This property will help us to prove the next lemma.

\begin{lemma}
\label{fiber}
    With the previous notations, if $\Sigma_{V\times\fn}$ satisfies the LARC, $A$ is a nilpotent matrix, and if there exists a nonzero $v\in\ker A$ such that, for any $x_1, x_2\in\inner D$, there exists $S>0$ and $u\in\UC$ such that 
    $$\phi(S, \langle v\rangle\times\{x_1\}, u)=\langle v\rangle\times\{x_2\},$$
    then, there exists $x\in\inner D$ such that
    $$\langle v\rangle\times \{x\}\subset\inner\OC^+(0, x)\cap \inner\OC^-(0, x).$$
\end{lemma}

\begin{proof} Let us start by noticing that
$$\langle v\rangle\times \{x\}\subset\inner\OC^-(0, x)\hspace{.5cm}\iff\hspace{.5cm} (0, x)\in \inner\OC^+(tv, x), \hspace{.5cm}\forall t\in\R$$
$$\stackrel{(\ref{sum})}{\iff}\hspace{.5cm} (0, x)\in (tv, 0)+\inner\OC^+(0, x)\hspace{.5cm}\forall t\in\R\hspace{.5cm}{\iff}\hspace{.5cm} \langle v\rangle\times \{x\}\subset\inner\OC^+(0, x),$$
and hence, it is enough to show the inclusion for the positive orbit. We divide the rest of the proof into four steps.

{\bf Step 1:} There exist $S>0$, $x_1, x_2\in\inner D$, and $\lambda_1, \lambda_2>0$ such that
$$(\lambda_1 v, x_1)\in\inner\OC^+(0, x_1)\hspace{.5cm}\mbox{ and } (-\lambda_2 v, x_2)\in\inner\OC^+(0, x_2)$$

Since $\Sigma_V$ is a linear control system satisfying the LARC, for any $S>0$, the map
$$u\in\UC\mapsto\phi^1(S, 0, u)\in V,$$
is a submersion around $\mathbf{0}\in \UC$ (see \cite[Theorem 3.5]{Tirao}). Let $S>0$ such that $\det(1-\rme^{S \DC_0})\neq 0$ and consider the open set $\UC_S^0$ defined in the previous section. Since the set
$$\left\{u\in\UC_S^0; f_{\varphi_{S, u}^B}^{-1}(\phi^A_{S, u}(0))\in\inner D\right\},$$
is dense in $\UC_S^0$ and $\phi^1(S, 0, \UC_S^0)$ is an open neighborhood of $0\in V$, there exist $u_1, u_2\in\UC_S^0$ such that
$$\phi^1(S, 0, u_1)=\lambda_1 v, \hspace{.5cm} \phi^1(S, 0, u_2)=-\lambda_2 v, \hspace{.5cm}\mbox{ with }\hspace{.5cm}\lambda_1, \lambda_2>0,
$$
and
$$x_i:=f_{\varphi_{S, u_i}^B}^{-1}(\phi^A_{S, u_i}(0))\in\inner D, \hspace{.5cm}i=1, 2.$$

Moreover, since $\Sigma_{V\times\fn}$ satisfies the LARC, if necessary, we can change $u_i$ by its concatenation with a control function satisfying (\ref{inner}), as done in the proof of Proposition \ref{dense}. Hence,
$$((-1)^{i+1}\lambda_i, x_i)=\phi(S, (0, x_i), u_i)\in\inner\OC^+(0, x_i),$$
concluding the proof of Step 1.

{\bf Step 2:} For $x_1, x_2\in N$ obtained in Step 1, there exists $T_0>0$ such that
$$\forall t> T_0, \hspace{.5cm}(tv, x_1)\in\inner\OC^+(0, x_1)\hspace{.5cm}\mbox{ and }\hspace{.5cm}(-tv, x_2)\in\inner\OC^+(0, x_2)$$

Since $(\lambda_1v, x_1)\in\inner\OC^+(0, x_1)$, there exists $0<a<b$ such that $\left((a, b)v, x_1\right)\subset\inner\OC^+(0, x_1)$. Therefore, for any $\delta_1, \delta_2\in (a, b)$, there exist $u_1, u_2\in\UC$ and $S_1, S_2>0$ such that
$$\phi(S_i, (0, x_1), u_i)=(\delta_i v, x_1)\hspace{.5cm}\implies\hspace{.5cm}\phi(S_2, \phi(S_1, (0, x_1), u_1), u_2)=\phi(S_2, (\delta_1 v, x_1), u_2)$$
$$=(\delta_1 v, 0)+\phi(S_2, (0, x_1), u_2)=(\delta_1 v, 0)+(\delta_1 v, x_1)=\left((\delta_1+\delta_2)v, x_1\right),$$
showing that $((na, nb)v, x_1)\in\inner\OC^+(0, x_1)$ for all $n\in\N$. Now,
$$n_0\geq \frac{b}{b-a}\hspace{.5cm}\iff\hspace{.5cm} (n_0-1)b\geq n_0a\hspace{.5cm}\implies\hspace{.5cm}\bigcup_{n\ge\left\lfloor\frac{b}{b-a}\right\rfloor}(na, nb)=\left(\left\lfloor\frac{b}{b-a}\right\rfloor a, +\infty\right),$$
implying that
$$\forall t> \left\lfloor\frac{b}{b-a}\right\rfloor a, \hspace{.5cm} (t v, x_1)\in\inner\OC^+(0, x_1).$$

In the same way, the fact that $(-\lambda_2, x_2)\in\inner\OC^+(0, x_2)$ implies the existence of $0<d<c$ such that $((-c, -d)v, x_2)\subset\inner\OC^+(0, x_2)$, which, as previously, allows us to conclude that
$$\forall t> \left\lfloor\frac{d}{c-d}\right\rfloor d, \hspace{.5cm} (-t v, x_2)\in\inner\OC^+(0, x_2).$$
Then,
$$T_0=\max\left\{\left\lfloor\frac{b}{b-a}\right\rfloor a, \left\lfloor\frac{d}{c-d}\right\rfloor d\right\},$$
satisfies the desired property, showing Step 2.

{\bf Step 3:} There exists $x\in\inner D$ and $T>0$ such that
$$|t|>T\hspace{.5cm}\implies\hspace{.5cm}(t v, x)\in\inner\OC^+(0, x).$$

Let $S_1, S_2>0$ and $u_1, u_2\in\UC$ such that
$$\phi(S_1, (0, x_2), u_1)=(a_1 v, x_1)\hspace{.5cm}\mbox{ and }\hspace{.5cm}\phi(S_2, (0, x_1), u_2)=(a_2 v, x_2),$$
which exist by our initial hypothesis. Let $a\in\R$, $S>0$, and $u\in\UC$ such that
$$(av, x_2)=\phi(S, (0, x_2), u)\in\inner\OC^+(0, x_2).$$
Now by Step 2, for any $t>T_0$, there exists $S_t>0$ and $u_t\in\UC$ such that $(tv, x_1)=\phi(S_t, (0, x_1), u_t)\in\inner\OC^+(0, x_1)$. Therefore, by concatenation, we get that
$$\inner\OC^+(0, x_2)\ni \phi(S_2, \phi(S_t, \phi(S_1, \phi(S, (0, x_2), u), u_1), u_t), u_2)$$
$$=(av, 0)+\phi(S_2, \phi(S_t, \phi(S_1, (0, x_2), u_1), u_t), u_2)=((a+a_1)v, 0)+\phi(S_2, \phi(S_t, (0, x_1), u_t), u_2)$$
$$((a+a_1+t)v, 0)+\phi(S_2, (0, x_1), u_2)=((a+a_1+a_2+t)v, x_2),$$
showing that
$$\forall t>T_0-a-a_1-a_2\hspace{.5cm} (tv, x_2)\in\inner\OC^+(0, x_2).$$
Therefore, $x=x_2$ and $T:=\max\{T_0, T_0-a-a_1-a_2\}$ satisfy the desired, showing Step 3.

{\bf Step 4:} There exists $x\in\inner D$ such that 
$$\langle v\rangle\times\{x\}\subset\inner\OC^+(0, x).$$

By Step 3, there exists $x\in\inner D$ and $t>0$ such that $(tv, x), (-tv, x)\in\inner\OC^+(0, x)$. By writing 
$$(tv, x)=\phi(S_1, (0, x), u_1)\hspace{.5cm}\mbox{ and }(tv, x)=\phi(S_2, (0, x), u_2),$$
allow us to conclude that
$$\inner\OC^+(0, x)\ni\phi(S_1, \phi(S_2, (0, x), u_2), u_1)=(tv, 0)+\phi(S_2, (0, x), u_1)=(tv, 0)+(-tv, x)=(0, x).$$
In particular, there exists $\delta>0$ such that
$$\forall t\in (-\delta, \delta), \hspace{.5cm}(tv, x)\inner\OC^+(0, x),$$
which, as in Step 2, implies that $(-n\delta, n\delta)v\times\{x\}\subset\inner\OC^+(0, x)$ for all $n\in\N$ and hence,
$$\langle v\rangle\times \{x\}=\bigcup_{n\in\N}(-n\delta, n\delta)v\times\{x\}\subset\inner\OC^+(0, x),$$
concluding the proof. 
\end{proof}

\bigskip
We can now prove the main result of this section.

\begin{theorem}
\label{controlproducto}
If the control-affine system $\Sigma_{V\times \fn}$ satisfies the LARC, $A$ is nilpotent, and $\DC_0$ is invertible, then the set
$$\pi_2^{-1}(D)=V\times D,$$
is a control set with a nonempty interior of $\Sigma_{V\times \fn}$.
\end{theorem}

\begin{proof}
Let us prove the result by induction on the dimension of $V$. If $\dim V=0$, then $\Sigma_{V\times\fn}=\Sigma_{(\fn, *)}$ and the result is true by Theorem \ref{existence}. Let us then consider $V$, with $\dim V=n+1$, and assume the result to hold true for any system on the Cartesian $W\times \fn$, where $\dim W\leq n$ and the system on $W$ is a linear control system whose associated drift is nilpotent.

Since $A$ is nilpotent, $\ker A\neq\{0\}$. Let $v\in \ker A$ be a nonzero vector. By invariance, the linear control system $\Sigma_V$ factors to a linear control system $\Sigma_{V/\langle v\rangle}$ on $V/\langle v\rangle$, whose associated derivation is also nilpotent. By the inductive hypothesis, $V/\langle v\rangle\times D$ is a control set with a nonempty interior for the system $\Sigma_{V/\langle v\rangle\times \fn}$. In particular, for any $x_1, x_2$ there exists $S>0$ and $u\in\UC$ such that
$$\phi(S, (0, x_1), u)=(av, x_2)\hspace{.5cm}\stackrel{(\ref{sum})}{\implies} \hspace{.5cm}\phi(S, \langle v\rangle\times \{x_1\}, u)=\langle v\rangle\times \{x_2\}.$$
By Lemma \ref{fiber}, there exists $x\in\inner D$ such that
$$\langle v\rangle\times \{x\}\subset\inner\OC^+(0, x)\cap \inner\OC^-(0, x),$$
and hence, there exists a control set $C$ of $\Sigma_{V\times \fn}$ with $\langle v\rangle\times \{x\}\subset\inner C$. Since the map $\pi:V\times \fn\rightarrow V/\langle v\rangle\times \fn$ is a conjugation between $\Sigma_{V\times \fn}$ and $\Sigma_{V/\langle v\rangle\times \fn}$ satisfying
$$\pi(0, x)\in V/\langle v\rangle\times \inner D=\inner \left(V/\langle v\rangle\times D\right) \hspace{.5cm}\mbox{ and }\hspace{.5cm}\pi^{-1}(\pi(0, x))=\langle v\rangle\times \{x\}\subset \inner C,$$
we get by Proposition \ref{conjugation} that $C=V\times D$, proving the result.
\end{proof}

%On the other hand, if $(u^*, y^*)\in\omega(u, xh)$, there exists $S_k\rightarrow+\infty$ such that
%$$\Phi_{S_k}(u, x)\rightarrow (u^*, y^*)\hspace{.5cm}\implies\hspace{.5cm}\phi(S_k, xh, u)\rightarrow y^*.$$
%By taking a subsequence if necessary, it holds that, $\varphi_{S_k}(h)\rightarrow h^*$ and hence, 
%$$\phi(S_k, x, u)=\phi(S_k, xh, u)\varphi_{S_k}(h)^{-1}\rightarrow y^*(h^*)^{-1}:=x^*,$$
%implying that $(u^*, x^*)\in\omega(u, x)$. Since $(u, x)$ is $S_0$-periodic and $x\in\inner C$, it holds that $(u^*, x^*)$ is an inner pair. In particular, there exists $S>0$ such that $\phi(S, x^*, u^*)\in \inner\OC^+_S(x^*)$ {\color{red} explicar melhor!}. Hence,
%$$\phi(S, y^*, u^*)=\phi(S, x^*, u^*)\varphi_S(h^*)\in\inner\OC_S^+(x^*)\varphi_S(h^*)=\inner\OC_S^+(x^*h^*)=\inner\OC_S^+(y^*),$$
%showing that $(u^*, y^*)$ is also an inner pair. Therefore, for any $h\in Z(N)$, the set $\omega(u, xh)$ consists of inner pairs. Since the systems satisfies LARC, Corollay 4.5.9 from \cite{FCWK} assures the existence of a control set $C_h$ such that $\pi_2(\omega(u, xh))\subset\inner C_h$, where $\pi_2:\UC\times G\rightarrow G$ is the projection onto the second component.

\subsection{The control set of a linear control system}

In this section we prove our main result, giving conditions for the existence of control sets with a nonempty interior. Namely, we prove the following:

\begin{theorem}
\label{maintheorem}
Let $\Sigma_G$ be a linear control system on a solvable Lie group $G$, and denote by $N$ the nilradical of
$G$. If $\Sigma_G$ satisfies the LARC and $N_0=N\cap G_0$ is a compact subgroup, then $\Sigma_G$ admits a (unique) control set $C$ with a nonempty interior. Moreover, it holds that $G_0\subset \overline{C}$.
\end{theorem}

\begin{proof}
The idea is basically to show that the existence of one such control set is equivalent to the existence of a control for a system on an appropriated Cartesian product as the one given in Section 3.2. The proof is done by a series of reductions as follows:

\begin{itemize}
\item {\bf Reduction to the case $N_0=\{e\}$:} As shown in the proof of Theorem \ref{uniqueness}, if $N_0$ is a compact subgroup, it is a normal subgroup of $G$. Since it is also $\varphi$-invariant, the LCS $\Sigma_G$ induces a LCS $\Sigma_{\widehat{G}}$ on $\widehat{G}=G/N_0$, satisfying $\widehat{N}_0=\{e\}$. Moreover, if $\Sigma_{\widehat{G}}$ admits a unique control set $D$ with a nonempty interior and satisfying $0\in\overline{D}$, then $C=\pi^{-1}(D)$ is, by Proposition \ref{compactfiber}, a control with a nonempty interior of $\Sigma_G$ satisfying $N_0=\pi^{-1}(e)\subset\pi^{-1}(\overline{D})=\overline{C}$. Moreover, $C$ is unique, since $D$ is unique.
\bigskip

\item {\bf Reduction to the case $G_0\times_{\rho}\fn$:} If $N_0=\{e\}$, the map $\psi:G_0\times_{\rho}\fn\rightarrow G$ defined in Lemma \ref{conj} is an isomorphism. In particular, it induces a LCS $\Sigma_{G_0\times_{\rho}\fn}$, which is equivalent to $\Sigma_G$. Therefore, we only have to show the existence and uniqueness of control sets with a nonempty interior for the LCS $\Sigma_{G_0\times_{\rho}\fn}$. In order to analyze this control system, let us obtain explicitly the expression of its coordinates.

First, by Lemma \ref{conj}, the vector field induced by $\XC$ on $G_0\times_{\rho}\fn$ has the form $\XC|_{G_0}\times\DC|_{\fn}$ with $\det\DC|_{\fn}\neq 0$. Moreover $\det\DC|_{\fn}\neq 0$ since we are assuming that $N_0=\{e\}$.

Therefore, in order to obtain $\Sigma_{G_0\times_{\rho}\fn}$ in coordinates, we only have to see the expression of a right-invariant vector field on $G_0\times_{\rho}\fn$. Since $T_{(e, 0)}(G_0\times_{\rho}\fn)=\fg_0\times\fn$,
any right-invariant vector field on $G_0\times_{\rho}\fn$ is given by $(Y, Z)^R(g, x)=(dR_{(g, x)})_{(e, 0)}\alpha'(0)$, where $\alpha:(-\varepsilon, \varepsilon)\rightarrow G_0\times_{\rho}\fn$ is any differentiable curve satisfying $\alpha(0)=(e, 0)$ and $\alpha'(0)=(Y, Z)$ with $(Y, Z)\in\fg_0\times\fn$. By considering the curve $\alpha(s):=(\rme^{sY}, sZ)$ we get that
$$R_{(g, x)}(\alpha(s)))=(\rme^{sY}, sZ)(g, x)=\left(\rme^{sY}g, sZ*\rho(\rme^{sY})x)\right).$$
Therefore, by the BCH formula,
$$sZ*\rho(\rme^{sY})x=sZ+\rho(\rme^{sY})x+\frac{1}{2}[sZ, \rho(\rme^{sY})x]+\frac{1}{12}\left([sZ, [sZ, \rho(\rme^{sY})x]]+[\rho(\rme^{sY})x, [\rho(\rme^{sY})x, sZ]]\right)+\cdots,$$
and by differentiation at $s=0$, we obtain that
$$
(Y, Z)(h, x)=\left(Y(h), Z(x)+(d\rho)_eYx\right).
$$
As a consequence, the LCS on $\Sigma_{G_0\times_{\rho}\fn}$ has the form
\begin{flalign*}
&&\, \left\{
\begin{array}{l}
\dot{g}=\XC(g)+\sum_{j=1}^mu_jY_j(g)\\
\dot{x}=\DC(u)x+\sum_{j=1}^mu_jZ_j(x),
\end{array}
\right. &&\hspace{-1cm}\left(\Sigma_{G_0\times_{\rho} \fn}\right),
\end{flalign*}
where $\DC(u)=\DC_0+\sum_{j=1}^m\DC_j$ for $\DC_0=\DC|_{\fn}$ and $\DC_j=(d\rho)_eY_j\in\mathrm{Der}(\fn)$.

\bigskip

\item {\bf Reduction to the case $V\times \fn$:} Assume $N_0=\{e\}$. By Lemma \ref{conj}, the subgroup $G_0$ is abelian. If $T$ is the maximal compact subgroup of $G_0$, then $V=G_0/T$ is simply connected and hence, up to isomorphisms, a vectorial space. Let us consider the subgroup $T\times_{\rho}\{0\}$ and note that
$$(g_1, X_1)(t, 0)=(g_2, X_2)\hspace{.5cm}\iff\hspace{.5cm}g_1t=g_2\hspace{.5cm}\mbox{ and }\hspace{.5cm}X_1=X_2,$$
showing that 
$$(G_0\times_{\rho}\fn)/(T\times_{\rho}\{0\})=G_0/T\times \fn=V\times \fn.$$
Since $T$ is invariant by automorphisms, the subgroup $T\times_{\rho}\{0\}$ is invariant by automorphisms in $\mathrm{Aut}(G_0)\times\mathrm{Aut}(\fn)$. Therefore, the LCS $\Sigma_{G_0\times_{\rho}\fn}$ obtained in the previous item induces a control-affine system on $V\times \fn$ through the canonical projection
$$\pi:G_0\times_{\rho}\fn\rightarrow V\times \fn, \hspace{.5cm}(g, x)\mapsto (gT, X).$$
Since the projection onto the first factor is a homomorphism, the LCS
$$\dot{g}=\XC(g)+\sum_{j=1}^mu_jY_j(g),$$
on $G_0$ projects to a LCS on $V$. However, any LCS on a vectorial space $V$ has the form
$$\dot{v}=Av+\sum_{j=1}^mu_jb_j,$$
for some $b_1, \ldots, b_m\in V$, and hence, we conclude that $\pi$ conjugates the LCS $\Sigma_{G_0\times_{\rho}\fn}$ to the control-affine system on $V\times\fn$ given by
\begin{flalign*}
&&\, \left\{
\begin{array}{l}
\dot{v}=Av+\sum_{j=1}^mu_jb_j\\
\dot{x}=\DC(u)x+\sum_{j=1}^mu_jZ_j(x)
\end{array}
\right. &&\hspace{-1cm}\left(\Sigma_{V\times \fn}\right).
\end{flalign*}
By construction, the system $\Sigma_{V\times\fn}$ satisfies the LARC, $A$ is a nilpotent matrix, and $\det\DC_0\neq 0$ . Therefore, by Theorem \ref{controlproducto}, $V\times D$ is a control set with a nonempty interior of $\Sigma_{V\times \fn}$, where $D\subset \fn$ is the control system $\Sigma_{(\fn, *)}^A$ obtained by the projection of $\Sigma_{V\times \fn}$ onto its second coordinate. Since $T\times_{\rho}\{0\}$ is a compact subgroup, Proposition \ref{compactfiber} implies that $C=\pi^{-1}(V\times D)$ is a control set with a nonempty interior of $\Sigma_{G_0\times_{\rho}\fn}$. Moreover, 
$$0\in \overline{D}\hspace{.5cm}\implies\hspace{.5cm} V\times\{0\}\subset V\times \overline{D}\hspace{.5cm}\implies\hspace{.5cm}G_0\times\{0\}=\pi^{-1}(V\times\{0\})\subset\overline{C},$$
concluding the proof.
\end{itemize}

\end{proof}

\end{document}